\documentclass{amsart}

\usepackage{amssymb,amsthm,amsmath,amsxtra,pstricks}
\usepackage[all]{xy}

\newcommand{\F}{\mathbb F}

\newcommand{\PP}{\mathbb P}

\newcommand{\Z}{\mathbb Z}

\DeclareMathOperator{\opchar}{char}
\DeclareMathOperator{\PGL}{PGL}

\numberwithin{equation}{section}

\theoremstyle{plain}

\newtheorem{prop}[equation]{Proposition}

\newtheorem{lem}[equation]{Lemma}

\newtheorem*{thmnostar}{Theorem}

\theoremstyle{remark}

\newtheorem{rmk}[equation]{Remark}

\begin{document}

\title[Nondegenerate curves over small fields]{Nondegenerate curves of low genus \\ over small finite fields}
\author{Wouter Castryck}
\date{\today}
\address{Katholieke Universiteit Leuven, Departement Wiskunde, Afdeling Algebra, Celestijnenlaan 200B, B-3001 Leuven (Heverlee), Belgium}
\email{wouter.castryck@gmail.com}
\author{John Voight}
\address{Department of Mathematics and Statistics, University of Vermont, 16 Colchester Ave, Burlington, VT 05401, USA}
\email{jvoight@gmail.com}

\begin{abstract}
In a previous paper, we proved that over a finite field $k$ of sufficiently large cardinality, all curves of genus at most $3$ over $k$ can be modeled by a bivariate Laurent polynomial that is nondegenerate with respect to its Newton polytope.  In this paper, we prove that there are exactly two curves of genus at most $3$ over a finite field that are \emph{not} nondegenerate, one over $\mathbb{F}_2$ and one over $\mathbb{F}_3$.  Both of these curves have remarkable extremal properties concerning the number of rational points over various extension fields.
\end{abstract}

\maketitle
%\noindent \emph{Subject classification:} 14M25, 14H10

Let $k$ be a perfect field with algebraic closure $\overline{k}$.
To a Laurent polynomial $f = \sum_{(i,j) \in \mathbb{Z}^2} c_{ij}
x^iy^j \in k[x^{\pm 1}, y^{\pm 1}]$, we associate its Newton
polytope $\Delta(f)$, the convex hull in $\mathbb{R}^2$
of the points $(i,j) \in \mathbb{Z}^2$ for which $c_{ij} \neq 0$.
An irreducible Laurent polynomial $f$ is called \emph{nondegenerate}
with respect to its Newton polytope if for all faces $\tau
\subset \Delta(f)$ (vertices, edges, and $\Delta(f)$ itself), the system of equations
\[ f|_\tau = x \frac{ \partial f|_\tau }{\partial x} = y \frac{ \partial f|_\tau }{\partial y} = 0 \tag{$*$} \]
has no solution in $\overline{k}^{* 2}$, where $f|_\tau = \sum_{(i,j) \in \mathbb{Z}^2 \cap \tau} c_{ij} x^iy^j$.

A curve $C$ over $k$ is called \emph{nondegenerate} if
it is birationally equivalent over $k$ to a curve defined by a Laurent
polynomial $f \in k[x^{\pm 1}, y^{\pm 1}]$ that
is nondegenerate with respect to its Newton polytope.
For such a curve, a vast amount of geometric information is encoded in the
combinatorics of $\Delta(f)$.  For example, the (geometric) genus of $C$ is equal to the number 
lattice points (points in $\Z^2$) lying in the interior of $\Delta(f)$.  Owing to this
connection, nondegenerate curves have become popular objects of study in
explicit algebraic geometry.  (See e.g.\ Batyrev \cite{Batyrev} and the introduction in our preceding work \cite{mainpaper} for further background and discussion.)

In a previous paper {\cite{mainpaper}}, we gave a partial answer to the natural question: \emph{Which curves are nondegenerate?}

\begin{thmnostar}
Let $C$ be a curve of genus $g$ over $k$. Suppose that one of these conditions holds:
\begin{enumerate}
  \item[(i)] $g = 0$;
  \item[(ii)] $g = 1$ and $C(k) \neq \emptyset$;
  \item[(iii)] $g = 2,3$, and either $17 \leq \#k < \infty$, or $\#k = \infty$ and $C(k) \neq \emptyset$;
  \item[(iv)] $g=4$ and $k=\overline{k}$.
\end{enumerate}
Then $C$ is nondegenerate.

If $g \geq 5$, then the locus $\mathcal{M}^\emph{nd}_g$ of
nondegenerate curves inside the coarse moduli space of curves of genus $g$ satisfies
$\dim \mathcal{M}^\emph{nd}_g = 2g + 1$, except for $g=7$ where $\dim \mathcal{M}^\emph{nd}_7 = 16$.
In particular, a generic curve of genus $g$ is nondegenerate if and only if $g \leq 4$.
\end{thmnostar}

Throughout the rest of this article, we assume that $k$ is a finite field, and we consider the cases excluded in condition (iii) above by the condition that $\#k \geq 17$.  Based on a number of preliminary experiments, we guessed \cite[Remark 7.2]{mainpaper} that this condition is superfluous.  In truth we have the following theorem, which constitutes the main result of this paper.

\begin{thmnostar}
Let $C$ be a curve of genus $g \leq 3$ over a finite field $k$.  Then $C$ is nondegenerate unless 
$k = \mathbb{F}_2$ or $k = \mathbb{F}_3$, and $C$ is birational to 
\begin{center}
$C_2$: $(x+y)^4=(xy)^2+xy(x+y+1)+(x+y+1)^2$ over $\F_2$, \\
$C_3$: $y^3-y=(x^2+1)^2$ over $\F_3$,
\end{center}
respectively. 
\end{thmnostar}

\noindent Both $C_2$ and $C_3$ have genus $3$. In particular, all curves of genus $2$ are nondegenerate.

Intriguingly, $C_2$ and $C_3$ have other remarkable properties: they obtain
an extremal number of rational points over certain extension fields of $\F_2$ and $\F_3$, respectively.

The paper is organized into four sections.  In Sections~\ref{reducingboundhyperell}--\ref{reducingboundquartic}, we refine the bound on $\#k$ which guarantees that a curve of genus $2$ or $3$ over $k$ is nondegenerate.  In Section~\ref{bruteforce}, we perform an exhaustive computation using the computer 
algebra system \textsf{Magma} \cite{Magma} to reduce the bound further. At the same time, we search the remaining finite fields $\F_2$ and $\F_3$ for curves that are not nondegenerate.  We conclude by discussing the extremal properties of the two resulting curves in Section~\ref{extremal}.

\section{Refining the bound for hyperelliptic curves} \label{reducingboundhyperell}

If $\opchar k$ is odd, then any hyperelliptic curve over $k$ is easily seen to be nondegenerate.  Indeed, it is well-known that a hyperelliptic curve of genus $g$ is birationally equivalent over $k$ to an affine curve of the form $y^2=p(x)$, where $p(x) \in k[x]$ is a squarefree polynomial of degree $2g+1$ or $2g+2$.  Then directly 
from the definition $(*)$, one sees that the polynomial $f(x,y)=y^2 - p(x)$ is nondegenerate with respect to its Newton polytope.

If instead $\opchar k = 2$, then a hyperelliptic curve of genus $g$ has an affine model of the
more general form
\begin{equation} \label{Weierstrass}
  y^2 + r(x)y = p(x)
\end{equation}
with $r(x) \in k[x]$ of degree at most $g + 1$, and $p(x) \in k[x]$ of degree at most $2g + 2$, and at least
$2g + 1$ if $\deg r(x) < g + 1$ (see Enge \cite[Theorem~7]{Enge}).
Moreover, such a model will not have any singularities in the affine plane; however, this condition is not enough to ensure that the defining polynomial $f(x,y)=y^2 + r(x)y + p(x)$ is nondegenerate with respect to its Newton polytope.

\begin{rmk}
There is a small erratum in our previous paper \cite[Section~5]{mainpaper}.  We write that one can always take $2\deg r(x) \leq \deg p(x)$ and $\deg p(x) \in \{2g + 1,2g + 2\}$ in (\ref{Weierstrass}).
This might however fail if $k = \F_2$ and the hyperelliptic curve $C$ has the property that the degree $2$ morphism $\pi: C \rightarrow \mathbb{P}^1$ is
\emph{completely split} over $k$, i.e., there are two distinct points in $C(k)$ above each point $0,1,\infty \in \mathbb{P}^1(k)$.  This erratum has no effect on any further statement in the paper \cite{mainpaper}.
\end{rmk}

The main result of this section is as follows.

\begin{prop} \label{yaysec1}
Let $C$ be a hyperelliptic curve of geometric genus $g \geq 2$ over a finite field $k$.  If $\#k$ is odd or $\#k \geq g+4$, then $C$ is nondegenerate.
\end{prop}

\begin{proof}
Let $\#k=q$.  By the above, we may assume that $q \geq 8$ is even and that
$C$ is given by an equation of type (\ref{Weierstrass}).  Let $f(x,y)=y^2+r(x)y+p(x)$.

First, we claim that after applying a birational transformation we may assume that $r(x)$ is a polynomial of degree $g+1$ with nonzero constant term.  Since $q \geq g + 4 > g + 1$, there is an $a \in k$ such that $r(x-a)$ has nonzero constant term, so replacing $x \leftarrow x-a$ we may assume $r(x)$ has nonzero constant term.  Then the transformed polynomial
\[ f'(x,y)=x^{2g+2}f(1/x,y/x^{g+1}) = y^2 + r'(x)y + p'(x), \]
which corresponds to applying the applying the $\mathbb{Z}$-affine map
\[ (X,Y) \mapsto (2g + 2 - X - (g + 1)Y, Y) \]
to the exponent vectors of $f(x,y)$, has the property that $\deg r'(x)=g+1$.  Making another substitution $x \leftarrow x - b$ then completes the argument.

Then using the definition $(*)$, a short case-by-case analysis of the possible Newton polytopes 
shows that if $p(x)$ is squarefree, then $f(x,y)=y^2 + r(x)y + p(x)$ is nondegenerate with respect to its Newton polytope.
For each $t(x) \in k[x]$ of
degree at most $g+1$, consider the change of variables $y \leftarrow y + t(x)$; then under this transformation we have
$p(x) \leftarrow p_t(x) = p(x) + r(x)t(x) + t(x)^2$ and $r(x)$ is unchanged (since $\opchar k=2$). We use a sieving argument to show that there exists
a choice of $t(x)$ such that $p_t(x)$ is squarefree.  Note we have $q^{g+2}$ choices for $t(x)$.

Suppose that $p_t$ is not squarefree.  Then $p_t(x)$ is divisible by the square of
a monic irreducible polynomial $v(x)$ of degree $m \leq g + 1$.  But note that if $v^2 \mid p_{t_1}$ and $v^2 \mid p_{t_2}$ for two choices $t_1,t_2$, then subtracting we have
\[ v^2 \mid \bigl(r(t_1+t_2) + t_1^2+t_2^2\bigr) = (t_1+t_2)(r+t_1+t_2). \]
Moreover, if $v$ divides each of these two factors then in fact $v \mid r$.

We are then led to consider two cases.  First, 
suppose that $v \nmid r$.  Then either $v^2 \mid (t_1+t_2)$ or $v^2 \mid (r+t_1+t_2)$.
Let $h=\lfloor (g+1)/2 \rfloor$. If $m=\deg v \leq h$, then by sieving we conclude that $v^2 \mid p_t$ for at most $2q^{g+1-2m+1}=2q^{g+2-2m}$ values of $t$.  On the other hand, if $m > h$ then $\deg v^2 > g+1$ so by sieving we now have $v \mid p_t$ for at most two values of $t$.  Since the number of monic irreducible polynomials of degree $m$ over $k$ is bounded by $q^m/m$, the number of values of $t$ such that $p_t$ is divisible by $v^2$ with $v \nmid r$ is at most
\begin{align*}
& q (2q^{g+2-2}) + \frac{q^2}{2}(2q^{g+2-4}) + \dots + \frac{q^h}{h}(2q^{g+2-2h})+ 2\frac{q^{h+1}}{h+1} + \dots + 2\frac{q^{g+1}}{g + 1} \\
& \qquad = 2 \left( q^{g+1}+\frac{q^g}{2} + \dots + \frac{q^{g+2-h}}{h} + \frac{q^{h+1}}{h+1} + \dots + \frac{q^{g+1}}{g + 1} \right) \\
& \qquad = \left( 2 + \frac{2}{g+1} \right) q^{g+1} + 2 \left( \sum_{i=2}^h \frac{q^{g + 2 - i}}{i}
+ \sum_{i=h+1}^g \frac{q^i}{i}  \right)\\
& \qquad \leq \left( 2 + \frac{2}{g+1} \right) q^{g+1} + 2 \frac{q^{g+1} - 1}{q-1} \qquad \text{(note $h \geq 1$)} \\
& \qquad \leq \left( 2 + \frac{2}{g+1} + \frac{2}{q-1} \right) q^{g+1}.
\end{align*}

Next, suppose that $v \mid r$.  Then in any case $v \mid (t_1+t_2)$, and hence there are at most $q^{g+1-m+1}$ values of $t$ such that $v^2 \mid p_t$.  Since $\deg r \leq g+1$, in the worst case $r$ splits into $g+1$ linear factors over $k$, and
we have at most $(g+1)q^{g+1}$ values of $t$ for which $p_t$ is divisible by $v^2$ for some $v \mid r$.

Putting these together, we can find a value of $t(x)$ such that $p_t(x)$ is squarefree if
\[
q^{g+2} > \left( g+ 3 + \frac{2}{g+1} + \frac{2}{q-1} \right)q^{g+1},
\]
which holds whenever $q \geq g + 4$, since $g \geq 2$ and $q \geq 8$.
\end{proof}

For our genera of interest $g=2$ and $g=3$, Proposition \ref{yaysec1} proves that all hyperelliptic curves are nondegenerate except possibly over $\mathbb{F}_2$ and $\mathbb{F}_4$.

\section{Refining the bound for plane quartics} \label{reducingboundquartic}

In this short section, we refine the bound as in Section~\ref{reducingboundhyperell} but now for plane quartics.

\begin{lem} \label{planequartics}
Let $C \subset \PP^2$ be a nonsingular plane quartic over a finite field $k$.  If $\#k \geq 7$, then $C$ is nondegenerate.
\end{lem}

\begin{proof}
Again analyzing the conditions of nondegeneracy \cite[Examples 1.5--1.6]{mainpaper}, we see that to prove that $C$ is nondegenerate it suffices to find three nonconcurrent $k$-rational lines in $\PP^2$
which are not tangent to $C$.  The projective
transformation which maps the three intersection points to the
coordinate points (and the lines to the coordinate lines) realizes $C$ as nondegenerate with respect
to a Newton polytope of the following type:
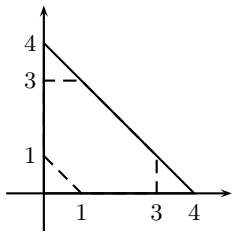
\begin{center}
\begin{pspicture}(-0.5,-0.5)(2.5,2.5)
\pspolygon[fillstyle=solid,linestyle=dashed](0.5,0)(1.5,0)(1.5,0.5)(0.5,1.5)(0,1.5)(0,0.5)
\pspolygon[fillstyle=none](0,0)(2,0)(0,2)
\psline{->}(-0.5,0)(2.5,0)
\psline{->}(0,-0.5)(0,2.5)
\rput(-0.18,2){\small $4$}
\rput(-0.18,1.5){\small $3$} \rput(-0.18,0.5){\small $1$}
\rput(2,-0.25){\small $4$} \rput(1.5,-0.25){\small $3$}
\rput(0.5,-0.25){\small $1$}
\end{pspicture}
\end{center}
(A dashed line appears as a face if our transformed curve contains the corresponding coordinate point.)

Write $m=\#C(k)$ and $q=\#k$. Since there are $q^2 + q + 1$ lines which are $k$-rational in $\mathbb{P}^2$, and
the number of $k$-rational lines through a fixed point is $q+1$, it suffices to prove that $C$ has strictly less than $q^2$ $k$-rational tangent lines.
%It is easy to see that, regardless of the nontangency condition, we have
%\[ (q^2 + q + 1)(q^2 + q + 1 - 1)(q^2 + q + 1 - (q+1)) = q^3(q + 1)(q^2 + q + 1) \]
%configurations of three nonconcurrent lines.

We claim that the number of $k$-rational tangent lines is at most $m+28$.
Of course each point of $C(k)$ determines a tangent line.
Suppose a $k$-rational line is tangent at a point of $C(\overline{k}) \setminus C(k)$; then
it is also tangent at each of the Galois conjugates of the point, which since $C$ is defined by a plane quartic immediately implies that the point is defined over a quadratic extension and
that the line is a bitangent.  By classical geometry and the theory of theta characteristics, there are at most $28$ bitangents (see e.g.\ Ritzenthaler \cite[Corollary 1]{Ritzenthaler}), and this proves the claim.

Thus if $q^2 > m+28$, we can find three nonconcurrent nontangent lines. By the Weil bound, it is sufficient that
\[ q^2 > q + 1 + 6\sqrt{q} + 28 \]
which holds whenever $q \geq 8$.  In fact, when $q=7$ then $m \leq 20$ by a result of Serre \cite{Serrenotes} (see also Top \cite{Top}), and so $q^2 > m+28$ for all $q \geq 7$.
\end{proof}

This lemma therefore proves that all plane quartics defined over finite fields are nondegenerate except possibly over $\mathbb{F}_q$ with $q\leq 5$.

\section{Computational results} \label{bruteforce}

From the results of the previous two sections, in order to prove our main theorem we perform an exhaustive computation in \textsf{Magma} to deal with the remaining cases:
\begin{enumerate}
  \item[(1a)] hyperelliptic curves of genus $g=2$ over $\mathbb{F}_2$ and $\mathbb{F}_4$;
  \item[(1b)] hyperelliptic curves of genus $g=3$ over $\mathbb{F}_2$ and $\mathbb{F}_4$;
  \item[(2)] nonsingular quartics in $\mathbb{P}^2$ over $\mathbb{F}_2$, $\mathbb{F}_3$, $\mathbb{F}_4$ and $\mathbb{F}_5$ (genus $g=3$).
\end{enumerate}
To this end, we essentially enumerated all irreducible polynomials whose Newton polytope is contained in
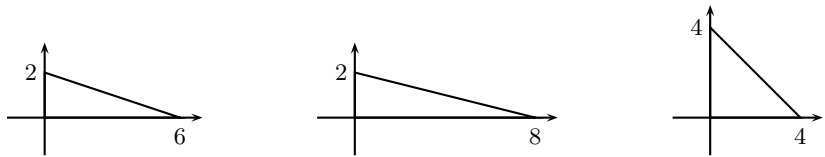
\begin{center}
\begin{pspicture}(-0.5,-0.5)(2.1,1.5)
\pspolygon[fillstyle=none](0,0)(1.8,0)(0,0.6)
\psline{->}(-0.5,0)(2.1,0)
\psline{->}(0,-0.5)(0,1)
\rput(-0.18,0.6){\small $2$}
\rput(1.8,-0.25){\small $6$}
\end{pspicture}
\qquad \qquad
\begin{pspicture}(-0.5,-0.5)(2.7,1.5)
\pspolygon[fillstyle=none](0,0)(2.4,0)(0,0.6)
\psline{->}(-0.5,0)(2.7,0)
\psline{->}(0,-0.5)(0,1)
\rput(-0.18,0.6){\small $2$}
\rput(2.4,-0.25){\small $8$}
\end{pspicture}
\qquad \qquad
\begin{pspicture}(-0.5,-0.5)(1.5,1.5)
\pspolygon[fillstyle=none](0,0)(1.2,0)(0,1.2)
\psline{->}(-0.5,0)(1.5,0)
\psline{->}(0,-0.5)(0,1.5)
\rput(-0.18,1.2){\small $4$}
\rput(1.2,-0.25){\small $4$}
\end{pspicture}
\end{center}
respectively, regardless of whether they define a curve of genus $g$ or not.  For each of these,
we checked whether the Newton polytope contained $g$ interior lattice points, since by Baker's inequality \cite[Theorem~4.1]{BeelenPellikaan} an irreducible Laurent polynomial $f \in k[x^{\pm 1},y^{\pm 1}]$ defines a curve whose (geometric) genus is at most the number of lattice points in the interior of $\Delta(f)$.

The polynomials $f$ that passed this test were then checked
for nondegeneracy with respect to the edges of $\Delta(f)$.  Checking nondegeneracy with respect to the edges boils down to checking squarefreeness of a number of univariate polynomials of small degree, which can be done very efficiently.  The nondegeneracy condition with respect to the vertices of $\Delta(f)$ is automatic.  The nondegeneracy condition with respect to $\Delta(f)$ itself is also automatic if $f$ defines a genus $g$ curve (by Baker's inequality), so we can disregard any polynomial for which this condition is not satisfied.

The polynomials $f$ that were \emph{not} nondegenerate with respect to the edges then saw
further investigation.  First, and only at this stage, we verified that in fact $f$ defines a curve of genus $g$.  Then, repeatedly, we applied a random transformation to $f$ of the following form:
\begin{enumerate}
  \item $(x,y) \leftarrow (x-a, y-h(x))$ for $a \in k$ and $h(x) \in k[x]$ of degree at most $g+1$ (for hyperelliptic curves);
  \item a projective linear transformation (for plane quartics).
\end{enumerate}
We then again checked the resulting polynomial for nondegeneracy with respect to the edges. Polynomials for which there were $1000$ failures in a row were stored in a list.

In each of the hyperelliptic curve cases the list remained empty, implying the following lemma.

\begin{lem}
All hyperelliptic curves of genus at most $3$ defined over a finite field are nondegenerate.
\end{lem}

In the plane quartic case, the list eventually contained exactly one polynomial for $k = \mathbb{F}_2$:
\[ f_2:(x+y)^4+(xy)^2+xy(x+y+1)+(x+y+1)^2. \]
We then tried \emph{all} projective linear transformations in $\PGL_3(\F_2)$ and found that, quite remarkably, $f_2$ is invariant under each of these transformations---the canonical embedding here is truly canonical!

Over $k=\mathbb{F}_3$, we were left with a set of polynomials that turned out to be all projectively equivalent to the polynomial
\[ f_3=y^3-y-(x^2+1)^2. \]
We exhaustively verified that none of the projectively equivalent polynomials is nondegenerate with
respect to its Newton polytope.

Over $\mathbb{F}_4$ and $\mathbb{F}_5$, the list remained empty.  We therefore have the following proposition.

\begin{prop}
Over any finite field $k$, all curves $C/k$ of genus at most $3$ are nondegenerate, except
if $k = \mathbb{F}_2$ and $C$ is $k$-birationally equivalent to $C_2$, or if $k = \F_3$ and
$C$ is $k$-birationally equivalent to $C_3$.
\end{prop}

\begin{proof}
It remains to show that if $C$ is a nonhyperelliptic curve of genus $3$ which 
can be modeled by a nondegenerate Laurent polynomial $f$, then it 
can be modeled by a nondegenerate Laurent polynomial whose Newton
polytope is contained in $4\Sigma$, the convex hull of the points $(0,0)$, $(0,4)$, and $(4,0)$. This is true
because $\Delta(f)$ has three interior lattice points which are not collinear, since $C$ is not hyperelliptic \cite[Lemma 5.1]{mainpaper}. Applying a $\mathbb{Z}$-affine transformation
to the exponent vectors, we may assume that in fact the interior lattice points of $\Delta(f)$ are $(1,1)$, $(1,2)$, and $(2,1)$.  But then $\Delta(f)$ is contained in the maximal polytope with these interior lattice points, which is $4\Sigma$ \cite[Lemma 10.2]{mainpaper}. The result follows.
\end{proof}

We conclude with a remark on the total complexity of the above computation. Since we are
only interested in curves up to birational equivalence, rather than simply enumerating all polynomials of a given form one could instead enumerate curves by their moduli.  Questions of this type in low genus have been pursued by many authors: Cardona, Nart, and Pujol\`as \cite{CNP} and Espinosa Garc\'ia, Hern\'andez Encinas, and Mu\~noz Masqu\'e \cite{EGHEMM} study genus $2$; Nart and Sadornil \cite{NartSad} study hyperelliptic curves of genus $3$; and Nart and Ritzenthaler \cite{NartRitz} study nonhyperelliptic curves of genus $3$ over fields of even characteristic.  In this paper we used a more naive approach since it is
more transparent, easier to implement, and at the same time still feasible.

We did however make use of the following speed-ups.  For hyperelliptic curves of genus $g=3$ with $\#k=4$, the  coefficient of $x^8$ and the constant term can always be taken $1$; for plane quartics with $\#k=4$, the coefficients of $x^4$ and $y^4$ and the constant term can always be taken $1$.  Finally, for plane quartics with $\#k=5$, from the proof of Lemma~\ref{planequartics}, we may assume that there exist at least two $k$-rational tangent lines that are only tangent over $\overline{k}$ (otherwise there exist enough nontangent lines to ensure nondegeneracy); transforming these to $x$- and $y$-axis, we may thus assume that $f(x,0) = (ax^2 + bx + 1)^2$ and $f(0,y) = (cy^2 + dy + 1)^2$ with $a,b,c,d \in k$.

\section{Extremal properties} \label{extremal}

In this section, let $C_2$ and $C_3$ denote the complete nonsingular models of the curves defined as in the main theorem.

The curve $C_2$ can be found in many places in the existing literature.  It enjoys some remarkable properties concerning the number $\#C_2(\mathbb{F}_{2^m})$ of $\mathbb{F}_{2^m}$-rational points for various values of $m$.  First, it has no $\mathbb{F}_2$-rational points. However, over $\mathbb{F}_4$ and $\mathbb{F}_8$ it has $14$ and $24$ points, respectively; in both cases, this is the maximal number of rational points possible on a complete nonsingular genus $3$ curve, and in each case $C_2$ is the unique curve obtaining this bound (up to isomorphism).  However, over $\mathbb{F}_{32}$ the curve becomes pointless again! And once more, it is the unique curve having this property.  For the details,
see Elkies \cite[Section~3.3]{Elkies}. We refer to work of Howe, Lauter, and Top \cite[Section~4]{HLT} for more on pointless curves of genus $3$.  It is remarkable that this curve is also distinguished by considering conditions of nondegeneracy.

In fact, $C_2$ is a twist of the reduction modulo $2$ of the Klein quartic (defined by the equation $x^3y + y^3z + z^3x=0$), which has more extremal properties. For instance,
Elkies \cite[Section~3.3]{Elkies} has shown that the Klein quartic modulo $3$ is extremal over fields of the form $\mathbb{F}_{9^m}$.  If $m$ is odd, its number of points is maximal.  If $m$ is even, its number of points is minimal.  Although the curve $C_3$ is
not isomorphic over $\overline{\mathbb{F}}_3$ to the Klein quartic, over $\mathbb{F}_{27}$ it
has the same characteristic polynomial of Frobenius, being $(T^2 + 27)^3$.  It follows that
$C_3$ shares the extremal properties of the Klein quartic over fields of the form $\mathbb{F}_{3^{6m}}$: $C_3$ has the maximal number of points possible if $m$ is odd, and the minimal number of points possible if
$m$ is even.

We conclude with the following question: Is there a hyperelliptic curve (of any genus) defined over a finite field which is not nondegenerate?  If so, it might also have interesting extremal properties.

\section*{Acknowledgements}

The first author would like to thank Alessandra Rigato for some helpful comments on curves over finite fields having many or few rational points.

\end{document}